\documentclass[10pt]{article}
\font\smallit=cmti10

\usepackage{amssymb,latexsym,amsmath,epsfig,amsthm} 
\usepackage{url}

\makeatletter

\renewcommand\section{\@startsection {section}{1}{\z@}
	{-30pt \@plus -1ex \@minus -.2ex}
	{2.3ex \@plus.2ex}
	{\normalfont\normalsize\bfseries\boldmath}}

\renewcommand\subsection{\@startsection{subsection}{2}{\z@}
	{-3.25ex\@plus -1ex \@minus -.2ex}
	{1.5ex \@plus .2ex}
	{\normalfont\normalsize\bfseries\boldmath}}

\renewcommand{\@seccntformat}[1]{\csname the#1\endcsname. }

\makeatother

\newtheorem{theorem}{Theorem}

\newtheorem{example}[theorem]{Example}

\theoremstyle{definition}

\newtheorem*{remark*}{Remark}

\newcommand{\Mod}[1]{\ (\mathrm{mod}\ #1)}
\newcommand{\Out}[1]{\mathrm{Out}(#1)}

\usepackage{xparse}   
\NewDocumentCommand{\ceil}{s O{} m}{%
	\IfBooleanTF{#1} 
	{\left\lceil#3\right\rceil} 
	{#2\lceil#3#2\rceil} 
}

\NewDocumentCommand{\floor}{s O{} m}{%
	\IfBooleanTF{#1} 
	{\left\lfloor#3\right\rfloor} 
	{#2\lfloor#3#2\rfloor} 
}

\begin{document}
	\begin{center}
		\uppercase{\bf Directed graphs of integers with arcs determined by an arithmetic function}
		\vskip 20pt
		{\bf Phakhinkon Napp Phunphayap}\\
		{\smallit Department of Mathematics, Faculty of Science,
			Burapha University, \\Chonburi, 20131, Thailand}\\
		{\tt phakhinkon.ph@go.buu.ac.th}, {\tt phakhinkon@gmail.com}\\
		\vskip 10pt
		{\bf Passawan Noppakaew\footnote{Passawan Noppakaew is the corresponding author.}}\\
		{\smallit Department of Mathematics, Faculty of Science, Silpakorn University, \\Nakhon Pathom, 73000, Thailand}\\
		{\tt noppakaew\_p@silpakorn.edu}\\
		\vskip 10pt
		{\bf Prapanpong Pongsriiam}\\
		{\smallit Department of Mathematics, Faculty of Science, Silpakorn University, \\Nakhon Pathom, 73000, Thailand\\
		\medskip
		and\\
		\medskip
		Graduate School of Mathematics, Nagoya University\\ Nagoya, 464-8602, Japan}\\
		{\tt pongsriiam\_p@silpakorn.edu}, {\tt prapanpong@gmail.com}
	\end{center}
	\vskip 20pt
	\centerline{\smallit Received: , Revised: , Accepted: , Published: } 
	\vskip 30pt
			
	\centerline{\bf Abstract}
	\noindent
	We introduce a new arc in directed graphs of integers. Among other things, we determine the positive integers that have arcs to all except a finite number of positive integers. We also propose some possible research problems at the end of this article.
	
	\pagestyle{myheadings}
	\thispagestyle{empty}
	\baselineskip=12.875pt
	\vskip 30pt
	
	\section{Introduction}
	A graph labelling is a problem in graph theory concerning an assignment of integers to vertices and/or edges of a graph under some certain conditions. For an up-to-date review of graph labelling, we refer the reader to Gallian \cite{Gallian}. Coprime graph of integers, which is the graph whose vertex set is the set of integers and two vertices $a$ and $b$ are connected by an edge if and only if $\gcd(a,b)=1$, provides an interesting  connection between graph theory and number theory, and many of their properties have been continuously studied by several mathematicians; see for example in the work of Erd\H{o}s \cite{Erdos1}, Erd\H{o}s and S\'{a}rk\"{o}zy \cite{Erdos2}, S\'{a}rk\"{o}zy \cite{Sarkozy}, Viadya and Prajapati \cite{Vaidya}, Berliner et al. \cite{Berliner}, Lee \cite{Lee}, and Berkove and Brilleslyper \cite{Berkove}. The study on graph labelling, coprime graphs, and a problem in TMO 2022 motivate us to introduce a directed graph whose vertices are positive integers and directed edges or arcs are determined by arithmetic functions. 
	
	Let $g$ be an arithmetic function and $n,u$ positive integers. We say that there exists a $g$-directed edge or $g$-arc from $n$ to $u$ if there exists a positive integer $N$ such that $N$ is divisible by $n$ and $g(N)=u$. In this case, we also say that there exists a $g$-arc from $n$ to $u$ and write $n \xrightarrow[]{g} u$. If there is no confusion and the function $g$ is understood, we sometimes drop the letter $g$ and write only $n\rightarrow u$ and say that there exists an arc from $n$ to $u$. We can think of this as a directed graph $G=(V,E_{g})$ where $V=\mathbb{N}$ and there is an arc from $n$ to $u$ if $n\xrightarrow{g}u$.
	
	Our purpose is to obtain some basic properties of this new arc between integers and find the integers $n$ that have $g$-arcs to as many integers as possible. So we let $\Out{g,n}$ be the set of all $u\in\mathbb{N}$ to which there exists a $g$-arc from $n$, that is, 
	\[\Out{g,n}=\{u\in\mathbb{N}\mid n\xrightarrow{g}u\},\]
	and we will determine the integers $n$ such that $\Out{g,n}$ is a cofinite subset of $\mathbb{N}$. The set $\text{In}(g,n)$ of $u\in\mathbb{N}$ such that $u\xrightarrow{g}n$ may be interesting too, but we postpone the investigation of these problems as a future project; see Questions 5 and 7 at the end of this article. Please see also Questions 6, 8, 9, 10 for some variations and different problems.
	
	To avoid some complications or trivialities, we restrict our attention to the $g$-arc where $g$ is an arithmetic function that is surjective or has the range as a cofinite subset of $\mathbb{N}\cup\{0\}$. In particular, we study this arc when $g=s_b,\tau,\omega,$ and $\Omega$, where $s_b (n)$ is the sum of digits of $n$ when $n$ is written in base $b\geq2$, $\tau(n)$ is the number of positive divisors of $n$, $\omega(n)$ is the number distinct prime divisors of $n$, and $\Omega(n)$ is the number of prime divisors of $n$ counted with multiplicity. Therefore the existence of our $g$-arc from $n$ to $u$ depends not only on the integers $n$ and $u$, but also on the property of the arithmetic function $g$. 
	
	We will propose some possible research problems on $g$-arc between integers at the end of this article too. Nevertheless, we do not claim that they are difficult or interesting. They may be unimportant and may even be trivial. However, we would merely like to record them for ourselves and share them among interested readers. For more advanced graph-theoretic problems, see for examples in the books by Bollob\'{a}s \cite{Bollobas} and Bondy and Murty \cite{Bondy}.
	
	\section{Main Results}
	
	Recall that if $b\geq2$ and $n\geq1$ are integers and 
	\begin{equation*}
		n=(a_ka_{k-1}\cdots a_0)_b=\sum_{i=0}^k a_ib^i
	\end{equation*}
	is the expansion of $n$ in base $b$ where $a_k\neq0$ and $0\leq a_i<b$ for all $i$, then we define the sum of digits of $n$ in base $b$ by $s_b(n)=a_1+a_2+\cdots+a_k$. So, for example, $s_{10}(123)=6$ and $s_5(123)=s_5((443)_5))=11$. We have the following result for $s_b$-arc.
	
	\begin{theorem}\label{thm1}
		Let $b\geq2$, $n\geq1$, and $d=\gcd(b-1,n)$. Then the following statements hold.
		\begin{itemize}
			\item [\rm{(i)}]\label{thm1(i)} $\Out{s_b,n}\subseteq \{u\in\mathbb{N} \mid u\equiv 0\Mod{d}\}$.
			\item [\rm{(ii)}] $\Out{s_b,n}$ is cofinite if and only if $d=1$.
		\end{itemize}
		
	\end{theorem}
	
	\begin{proof}
		In this proof, we write $n\rightarrow u$ instead of $n\xrightarrow{s_b}u$ for convenience. We first prove (i).
		So let $u\in\Out{s_b,n}$. Then there exists $N\in\mathbb{N}$ such that $n\mid N$ and $s_b(N)=u$. Since $b^k\equiv1\Mod{b-1}$ for every $k\in\mathbb{N}\cup\{0\}$,
		we see that $s_b(N)\equiv N\Mod{b-1}$. Since $d\mid b-1$, we obtain $u=s_b(N)\equiv N\Mod{d}$. In addition, we have $d\mid n$ and $n\mid N$, and therefore $d\mid N$ and $u\equiv N\equiv0\Mod{d}$. This proves (i). 
		
		If $d>1$, it follows immediately from (i) that $\Out{s_b,n}$ is not cofinite. So to prove (ii), it is enough to consider the case that $d=1$. If $n=1$, then for any $u\in\mathbb{N}$, we can choose $N=\sum_{i=0}^{u-1} b^i$ to obtain $n\mid N$ and $s_b(N)=u$, and so $\Out{s_b,n}=\mathbb{N}$.  Next, let 
		$n\geq2$, $b=p_1 ^{a_1}p_2 ^{a_2} \cdots p_k^{a_k}$, and $n=p_1 ^{n_1} p_2^{n_2}\cdots p_k^{n_k}m$, where $p_1,p_2,\cdots,p_k$ are distinct primes, $a_1,a_2,\cdots,a_k,m$ are positive integers, $n_1,n_2,\cdots,n_k$ are nonnegative integers, and $(m,b)=1$. If $b=2$ and $m=1$, then $n$ is a power of $2$, and for any $u\in\mathbb{N}$, we can choose $N=n\sum_{i=0}^{u-1} 2^i$ to obtain $N\equiv0\Mod{n}$ and $s_2(N)=u$. This shows that if $b=2$ and $m=1$, then $\Out{s_b,n}=\mathbb{N}$. Here we record the case that $\Out{s_b,n}=\mathbb{N}$ for future reference. We have proved that
		\begin{equation}\label{2}
			\text{$\Out{s_b,n}=\mathbb{N}$ if $n=1$ or if $b=2$ and $m=1$.}
		\end{equation}
		So from this point on, we assume that $n\geq2$ and if $b=2$ then $m>1$. Next, let
		\begin{center}
			$c=\max\limits_{1\leq j\leq k} n_j$,\quad $A=b^c\sum\limits_{j=1}^n  b^{j\phi(m)}$,
			
			$B=b^c\left(\sum\limits_{j=1}^{n-1}b^{j\phi(m)} + \left\lfloor\dfrac{b}{2}\right\rfloor b^{n\phi(m)-1} + \left\lceil\dfrac{b}{2} \right\rceil b^{(n+1)\phi(m)-1}   \right)$,
		\end{center}
		  where $\phi$ is the Euler totient function, $\floor{x}$ is the largest integer not exceeding $x$ and $\ceil{x}$ is the smallest integer larger than or equal to $x$. We assert that $A$ and $B$ have the following properties:
		  \begin{equation}
		  	A\equiv0\Mod{n},B\equiv 0\Mod{n}, \text{ and } \gcd(s_b(A),s_b(B))=1.
		  \end{equation}  
	  	It is clear that $A$ and $B$ are divisible by $p_1^{n_1}p_2^{n_2}\cdots p_k^{n_k}$. In addition, by Euler's theorem and the fact that $(b,m)=1$, we obtain $b^{\phi(m)}\equiv1\Mod{m}$ and therefore
		\begin{equation}\label{eq2}
			A\equiv nb^c\Mod{m}\text{ and } bB\equiv b^c\left(b(n-1)   + \left\lfloor\dfrac{b}{2}\right\rfloor + \left\lceil\dfrac{b}{2} \right\rceil  \right)\Mod{m}.
		\end{equation}
		It is easy to verify that
		\begin{equation*}
			\left\lceil\dfrac{b}{2} \right\rceil + \left\lfloor\dfrac{b}{2}\right\rfloor =b.
		\end{equation*} 
		From this, \eqref{eq2}, and the fact that $m\mid n$ we obtain 
		\begin{equation}\label{eq1}
			\text{$A\equiv 0 \Mod{m}$ and $bB\equiv nb^{c+1}\equiv0\Mod{m}$.}
		\end{equation}
		Since $(b,m)=1$, we obtain $B\equiv0\Mod{m}$. Therefore both $A$ and $B$ are divisible by $p_1^{n_1}p_2^{n_2}\cdots p_k^{n_k}$ and also by $m$, so they are divisible by $n$. It remains to show that $\gcd (s_b(A),s_b(B))=1$. First, it is clear that $s_b(A)=n$. Nevertheless, we have to be more careful in the calculation of $s_b(B)$ since the term $b^{(n-1)\phi(m)}$ and $\left\lfloor b/2 \right\rfloor b^{n\phi(m)-1}$ may correspond to the same position in the $b$-adic expansion of $B$. Recall that \eqref{2} is already proved and we assume that if $b=2$, then $m>1$. Since $(m,b)=1$, we see that if $b=2$, then $m\geq3$. This implies that either $b\geq3$ or $\phi(m)>1$. If $b\geq 3$, then
		\begin{equation*}
			\text{ $1+ \left\lfloor\dfrac{b}{2}\right\rfloor\leq 1+\dfrac{b}{2}<b$, and so $s_b(B)=n-1+\left\lfloor\dfrac{b}{2}\right\rfloor+\left\lceil\dfrac{b}{2} \right\rceil=n-1+b$.}
		\end{equation*} 	
		If $\phi(m)>1$, then	
		\begin{equation*}
			\text{ $(n-1)\phi(m)<n\phi(m)-1$, and so $s_b(B)=n-1+ \left\lfloor\dfrac{b}{2}\right\rfloor+\left\lceil\dfrac{b}{2} \right\rceil  =n-1+b$.}
		\end{equation*}
		In any case, $s_b(B)=n-1+b$, and so $\gcd(s_b(A),s_b(B))=\gcd(n,n-1+b)=d=1$. This proves \eqref{eq1}.
		
		Recall that the Frobenius number of coprime positive integers $a$ and $\ell$ is 
		\[(a-1)(\ell-1)-1,\]
		 that is, if $a,\ell\in\mathbb{N}$ and $\gcd(a,\ell)=1$, then for every positive integer $n\geq(a-1)(\ell-1)$, there exist $x,y \in \mathbb{N}\cup\{0\}$, such that $n=ax+\ell y$. Let $a=s_b(A)$, $\ell=s_b(B)$. Then $a$ and $\ell$ are coprime. We assert that there is an arc from $n$ to each  positive integer $u\geq(a-1)(\ell-1)$. We first write such an integer $u=ax+\ell y$ for some $x,y\in\mathbb{N}\cup\{0\}$. Then we construct $N$ as the concatenation of $A$ and $B$ as 
		\begin{equation}\label{eq3}
			N=(AAA\cdots ABBB\cdots B)_b,
		\end{equation}
		where the number of $A$ in \eqref{eq3} is $x$ and the number of $B$ in \eqref{eq3} is $y$. Since both $A,B\equiv0\Mod{n}$, we have $N\equiv0\Mod{n}$. In addition,  
		\[s_b(N)=xs_b(A)+ys_b(B)=ax+\ell y=u.\] 
		This proves our assertion. In other words,
		\begin{equation*}
			\text{$\Out{s_b,n}$ contains every integer  $ u\geq(a-1)(\ell-1)$.} 
		\end{equation*}
		Hence $\Out{s_b,n}$ is a cofinite subset of  $\mathbb{N}$, as required. This completes the proof.
	\end{proof}
	
	Although the Frobenius number formula for two coprime positive integers that we used in the proof of Theorem \ref{thm1} is well known and is not difficult to prove, finding a general Frobenius number formula for three integers $a_1, a_2$ and, $a_3$ with $\gcd(a_1, a_2, a_3)=1$ is not easy. In fact, it was unknown for quite some time until Tripathi \cite{Tripathi,Tripathi1} studied this problem in his thesis and published the results in 2017.
	
	From Theorem \ref{thm1}, the following natural questions may occur: is the set $\Out{s_b,n}$ always infinite? In addition if $(b-1,n)=1$, we see from \eqref{2} that $\Out{s_b,n}=\mathbb{N}$ when $n=1$ or when $b=2$ and $m=1$. Can we find a simple necessary and sufficient condition for $\Out{s_b,n}=\mathbb{N}$? We have a complete answer to these questions in the next two theorems.
	
	\begin{theorem}\label{THM2}
		For integers $b\geq2$ and $n\geq1$, $\Out{s_b,n}$ is an infinite set. More precisely, $\Out{s_b,n}$ contains every multiple of $s_b(n)$.
	\end{theorem}
	\begin{proof}
		If $u=ks_b(n)$ for some positive integer $k$, then the integer $N=(nn\cdots n)_b$ constructed by the concatenation of $k$ terms of $n$ satisfies $n\mid N$ and 
		\[s_b(N)=ks_b(n)=u.\] 
		So every multiple of $s_b(n)$ is contained in $\Out{s_b,n}$, as required.
	\end{proof}

	\begin{theorem}\label{THM3}
		Let $b\geq2$ and $n\geq1$ be integers. Then $\Out{s_b,n}=\mathbb{N}$ if and only if $(b-1,n)=1$ and every prime divisor of $n$ is also a divisor of $b$.
	\end{theorem}
	
		\noindent {\it Remark.} If $n=1$, then $(b-1,n)=1$ and $n$ has no prime divisor, so it is included in the case $\Out{s_b,n}=\mathbb{N}$. Nevertheless, for clarity, if one prefer, the above theorem can be rewritten as $\Out{s_b,n}=\mathbb{N}$ if and only if (i) $n=1$ or (ii) $n\geq2$, $(b-1,n)=1$, and every prime dividing $n$ is a divisor of $b$.
	
	\begin{proof}
		We first observe that if $(b-1,n)>1$, then we obtain by Theorem \ref{thm1} that $\Out{s_b,n}\neq \mathbb{N}$. Next, we consider the case that there exists a prime $p$ dividing $n$ and $p\nmid b$. Suppose for a contradiction that $\Out{s_b,n}= \mathbb{N}$. So in particular, $1\in \Out{s_b,n}$. Then there exists $N\in\mathbb{N}$ such that $n\mid N$ and $s_b(N)=1$. But $s_b(N)=1$ implies that $N=b^\ell$ for some $\ell \in\mathbb{N}\cup\{0\}$. Since $p\mid n$ and $n\mid N$, we have $p\mid b^{\ell}$, which implies $p\mid b$, a contradiction. Hence $\Out{s_b,n}\neq\mathbb{N}$, as required.
		
		For the converse, assume that $(b-1,n)=1$ and every prime divisor of $n$ is also a divisor of $b$. Since we already proved this result when $n=1$ in \eqref{2}, we can suppose that $n\geq2$. Let $b=p_1^{b_1}p_2^{b_2}\cdots p_k^{b_k}$ where $p_1,p_2,\ldots,p_k$ are distinct primes and $b_1,b_2,\ldots,b_k$ are positive integers. Since every prime divisor of $n$ is also a divisor of $b$, we can write $n=p_1^{n_1}p_2^{n_2}\cdots p_k^{n_k}$ where $n_1,n_2,\ldots,n_k$ are nonnegative integers. To show that $\Out{s_b,n}=\mathbb{N}$, let $u\in\mathbb{N}$ be given. Let $n_0=\max_{1\leq i\leq k} n_i$ and $N=b^{n_0}\sum_{i=0}^{u-1}b^i$. Then $n\mid N$ and $s_b(N)=u$. So $n\rightarrow u$, as required. This completes the proof.
	\end{proof}
	
	Another question that may occur from Theorem \ref{thm1} is in the subset relation in (i). By Theorem \ref{THM3}, we obtain that if $(b-1,n)=1$ and every prime divisor of $n$ is a divisor of $b$, then the subset relation in (i) of Theorem \ref{thm1} becomes an equality. If $(b-1,n)>1$, is the subset relation a strict subset or an equality? The next two examples show that it may be an equality or a strict subset relation.
	
	\begin{example}	\label{Ex4}
		Let $b=10$ and $n=33$. In this case, we have $d=\gcd(b-1,n)=3$ and we will show that $u=3$ is not an element of $\Out{s_b,n}$. Suppose by way of contradiction that there exists an arc from $n$  to $3$. Then there exists $N\in\mathbb{N}$ such that $n\mid N$ and $s_{10}(N)=3$. Since $11\mid n$, we see that $11\mid N$. Since $s_{10}(N)=3$, it is easy to see that $N$ must be in the following form: 
		\begin{center}
			$N=3\times 10^a$, or $N=(2\times 10^a)+10^b$, or $N=10^a+10^b+10^c$,
		\end{center}
		 where $a,b,c\in\mathbb{N}\cup\{0\}$. Reducing these modulo $11$, we obtain
		\begin{equation*}
			N\equiv3(-1)^a, 2(-1)^a+(-1)^b, (-1)^a+(-1)^b+(-1)^c\equiv 1, -1, 3, -3 \Mod{11},
		\end{equation*}
		which contradicts the fact that $11\mid N$. Hence $\Out{s_b,n}\neq\{u\in\mathbb{N} \mid u\equiv0\Mod{d}\}$.
	\end{example}
	
	There is a situation where $(b-1,n)>1$ and the subset relation in (i) of Theorem \ref{thm1}  is in fact an equality too as shown in the next example.
	
	\begin{example}\label{Ex5}
		Let $b=10$ and $n=3$. In this case, we have $d=\gcd(b-1,n)=3$. To show that the subset relation in Theorem \ref{thm1} (i) can be replaced by an equality, let $u\equiv0\Mod{3}$. Then $u=3q$ for some $q\geq1$. Then we can choose $N=3\sum_{i=0}^{q-1}10^i$ to obtain $N\equiv0 \Mod{n}$ and $s_{10}(N)=3q=u$, as required.
	\end{example}
	
	From this, it is interesting to determine exactly when the subset relation in Theorem \ref{thm1} is strict or is an equality. See Question 1 and the list of other questions at the end of this article.
	
	Recall that for each $m\in\mathbb{N}$, the number of positive divisors of $m$ is denoted by $\tau(m)$ and the arithmetic function $\tau$ is called the divisor function. It is well known that $\tau$ is a multiplicative function, that is $\tau(1)=1$ and if $m,n\in\mathbb{N}$ and $(m,n)=1$, then $\tau(mn)=\tau(m)\tau(n)$. Next, we study some properties of the $\tau$-arc and $\Out{\tau,n}$.
	
	\begin{theorem}\label{THM6}
		Let $n$ be a positive integer. Then the following statements hold.
		\begin{itemize}
			\item [\rm{(i)}] $\Out{\tau,n}\subseteq\left\{u\in\mathbb{N}\ \vert\  u\geq \tau(n) \right\}$.
			\item [\rm{(ii)}] $\Out{\tau,n}=\mathbb{N}$ if and only if $n=1$.
			\item [\rm{(iii)}] $\Out{\tau,n}$ is a cofinite subset of $\mathbb{N}$ if and only if $n=p^k$ for some prime $p$ and nonnegative integer $k$.
			\item[\rm{(iv)}] $\Out{\tau,n}$ is always an infinite set.
		\end{itemize}
	\end{theorem}  
	\begin{proof}
		For (i), let $u\in\Out{\tau,n}$. Then there exists $N\in\mathbb{N}$ such that $n\mid N$ and $\tau(N)=u$. Since $n\mid N$, every divisor of $n$ is also divisor of $N$. Therefore $\tau(n)\leq\tau(N)=u$, as required.
		
		For (ii), if $n=1$ and $u$ is a positive integer, then we can choose $N=p^{u-1}$ to obtain $n\mid N$ and $\tau(N)=u$, and so $\Out{\tau,n}=\mathbb{N}$. We observe that $\tau(n)=1$ if and only if $n=1$. So if $\Out{\tau,n}=\mathbb{N}$, then we obtain from (i) that $\tau(n)=1$, which implies $n=1$.
		
		For (iii), assume that $\Out{\tau,n}$ is cofinite. Since the number of primes is infinite, there exists a prime number $u\in \Out{\tau,n}$. So there is $N\in\mathbb{N}$ such that $n\mid N$ and $\tau(N)=u$. Let $N=p_1^{a_1}p_2^{a_2}\cdots p_k^{a_k}$ where $a_1,a_2,\ldots,a_k$ are positive integers and $p_1,p_2,\ldots,p_k$ are distinct primes. By the well known formula for $\tau(N)$, we obtain
		\begin{equation}\label{eq31}
			u=\tau(N)=(a_1+1)(a_2+1)\cdots(a_k+1).
		\end{equation} 
		Since $u$ is a prime, the only positive divisors of $u$ is $1$ and $u$. Therefore \eqref{eq31} implies that $k=1$ and $a_1=u-1$. So $N=p_1^{u-1}$. Since $n\mid N$, we see that $n=p_1^k$ for some nonnegative integer $k\leq u-1.$
		Next, let $n=p^k$ where $p$ is a prime and $k$ is a nonnegative integer. If $k=0$, then the result follows from (ii). So suppose $k\geq 1$. We assert that
		\begin{equation}\label{100}
			\Out{\tau,n}=\left\{u\in\mathbb{N}\ \vert\ u\geq k+1\right\}.
		\end{equation}
		By (i), we see that $\Out{\tau,n}$ is a subset of the set on the right-hand side of \eqref{100}. Next, $u\in\mathbb{N}$ and $u\geq k+1$. Let $N=p^{u-1}$. Then $n\mid N$ and $\tau(N)=u$. So $u\in\Out{\tau,n}$ and \eqref{100} is proved. Therefore $\Out{\tau,n}$ is cofinite. This proves (iii).
				
		Next, we prove (iv). If $n=1$, then (iv) follows from (ii). So let $n>1$ and write $n=p_1^{a_1}p_2^{a_2}\cdots p_k^{a_k}$ where $a_1,a_2,\ldots,a_k$ are positive integers and $p_1,p_2,\ldots,p_k$ are distinct primes. We assert that the integer $u$ defined by
		\begin{equation*}
			u=u_{\ell}=2^{\ell}(a_1+1)(a_2+1)\cdots(a_k+1)
		\end{equation*}
		 is an element of $\Out{\tau,n}$ for every $\ell\in\mathbb{N}$. Let $q_1,q_2,\ldots,q_{\ell}$ be distinct primes and different from $p_1,p_2,\ldots,p_k$. Let $m=\prod_{i=1}^{\ell}q_i$ and $N=mn$. Then $n\mid N$, and since $(m,n)=1$ and $\tau$ is multiplicative, we also obtain \[\tau(N)=\tau(m)\tau(n)=2^{\ell}(a_1+1)(a_2+1)\cdots(a_k+1)=u_{\ell}.\] 
		 So $u_{\ell}\in\Out{\tau,n}$. 
		 Since $\ell$ is arbitrary, $\Out{\tau,n}$ is infinite. This completes the proof.
	\end{proof}
	
	Considering (i) of Theorem \ref{THM6}, it is natural to ask whether the subset relation can be replaced by an equality. The next theorem shows that it is an equality if and only if $n=1$ or $n$ is a prime power.
	
	\begin{theorem}
		Let $n$ be a positive integer. Then $\Out{\tau,n}=\left\{u\in\mathbb{N}\ \vert\ u\geq \tau(n)\right\}$ if and only if $n=p^k$ for some prime $p$ and nonnegative integer $k$.
	\end{theorem}
	\begin{proof}
		If $n=1$, then the result follows from Theorem \ref{THM6}. If $n=p^k$ where $p$ is a prime and $k\in\mathbb{N}$, then we already proved it in \eqref{100}. So the converse of this theorem holds. Next, suppose $n\neq 1$ and $n$ is not a prime power. Then $n=p_1^{n_1}p_2^{n_2}\cdots p_k^{n_k}$ where $k\geq 2$, $n_1,n_2,\ldots,n_k$ are positive integers and $p_1, p_2,\ldots,p_k$ are distinct primes. We only need to find $u\in\mathbb{N}$ such that $u\geq \tau(n)$ and $u\notin\Out{\tau,n}$. Suppose, by way of contradiction, that $u=\tau(n)+1$ is an element of $\Out{\tau,n}$. Then there exists $q\in\mathbb{N}$ such that $\tau(nq)=u$. Clearly $q\geq 2$. So if $(q,n)=1$, then $u=\tau(n)\tau(q)\geq 2\tau(n)>\tau(n)+1$, which is not the case. So $(q,n)> 1$. Let $d=(q,n)$. Since $d>1$ and $d\mid n$, we write $d=p_1^{d_1}p_2^{d_2}\cdots p_k^{d_k}$ where $d_1, d_2,\ldots, d_k$ are nonnegative integers and there is at least one $j=1,2,\ldots, k$ such that $d_j\geq 1$. Since $nd\mid nq$, we see that $\tau(nd)\leq\tau(nq)$. Therefore
		\begin{align*}
			u=\tau(nq)\geq\tau(nd)&=\prod\limits_{i=1}^{k}\left(n_i+d_i+1\right)\\
			&=\prod\limits_{i=1}^{k}\left(n_i+1\right)\prod\limits_{i=1}^{k}\left(1+\dfrac{d_i}{n_i+1}\right)\\
			&=\tau(n)\prod\limits_{i=1}^{k}\left(1+\dfrac{d_i}{n_i+1}\right)\\
			&\geq\tau(n)\left(1+\dfrac{d_j}{n_j+1}\right)\\
			&\geq\tau(n)\left(1+\dfrac{1}{n_j+1}\right)
			>\tau(n)\left(1+\dfrac{1}{\tau(n)}\right)=\tau(n)+1,
		\end{align*}
	which is a contradiction. Hence $u=\tau(n)+1$ is not an element  of $\Out{\tau,n}$. This  completes the proof.
	\end{proof}
	
	Next, we give some results on $\Out{\omega,n}$ and $\Out{\Omega,n}$.
	
	\begin{theorem}\label{thm7}
		The set $\Out{\omega,n}$ is cofinite for every $n\in\mathbb{N}$. More precisely, we have
		\begin{equation} \label{31}
			\Out{\omega,n}=\left\{u\in\mathbb{N}\ \vert\ u\geq \omega(n)\right\}.
		\end{equation}
		Consequently, $\Out{\omega,n}=\mathbb{N}$ if and only if $n=p^k$ for some prime $p$ and integer $k\geq0$.
	\end{theorem} 
	
	\begin{proof}
		To prove \eqref{31}, let $u\in \Out{\omega,n}$. Then there exists $N\in\mathbb{N}$ such that $n\mid N$ and $\omega(N)=u$. Since $n\mid N$, we have $u=\omega(N)\geq \omega(n)$. Next, suppose $u\geq\omega(n)$. If $u=\omega(n)$, then we can choose $N=n$ to obtain $n\mid N$ and $\omega(N)=u$. So suppose $u=\omega(n)+\ell$ for some $\ell \in\mathbb{N}$. Since the number of primes is infinite, there are primes $p_1>p_2>\cdots>p_{\ell}$ that are not the divisors of $n$. Let $N=np_1 p_2 \cdots p_{\ell}$. Then $n\mid N$ and $\omega(N)=\omega(n)+\ell=u$. This proves \eqref{31}. If $n=p^k$ where $p$ is a prime and $k$ is a nonnegative integer, then $\omega(n)=0$ or $\omega(n)=1$, so we obtain from \eqref{31} that $\Out{\omega,n}=\mathbb{N}$. If $n\neq1$ and $n$ is not a prime power, then $\omega(n)\geq2$, and we obtain from \eqref{31} that $\Out{\omega,n}$ is not $\mathbb{N}$. So the proof is complete.  
	\end{proof}
	
	\begin{theorem}
		For each $n\in\mathbb{N}$, we have
		\begin{equation}\label{32}
			\Out{\Omega,n}= \left\{u\in\mathbb{N}\ \vert\ u\geq \Omega(n)\right\}.
		\end{equation}	
		Consequently, $\Out{\Omega,n}$ is cofinite for every $n\in\mathbb{N}$, and $\Out{\Omega,n}=\mathbb{N}$ if and only if $n=1$ or $n$ is a prime.
	\end{theorem}
	
	\begin{proof}
		The proof is similar to that of Theorem \ref{thm7}, so we skip some details. If $u\in\Out{\Omega,n}$, then there exists $N\in\mathbb{N}$ such that $n\mid N$ and $\Omega(N)=u$, and so $u=\Omega(N)\geq\Omega(n)$. Next, let $u\geq\Omega(n)$. If $u=\Omega(n)$, then we can choose $N=n$; if $u=\Omega(n)+\ell$ for some $\ell\in\mathbb{N}$, then we choose $N=np_1p_2\cdots p_{\ell}$ where $p_1,p_2,\ldots,p_{\ell}$ are distinct primes that does not divide $n$, which leads to the conclusion that $u\in\Out{\Omega,n}$. This proves \eqref{32}. We observe that $\Omega(n)\in\left\{0,1\right\}$ if and only if $n=1$ or $n$ is a prime. Then the other part of this theorem follows in the same way as in the proof of Theorem \ref{thm7}. So the proof is complete.
	\end{proof}

	\section{Comments and Some Open Questions}
	In this section, we give a list of some open problems. However, we do not claim that they are difficult or interesting. They may be unimportant or may even be trivial. However, we would merely like to record them for ourselves and share them among interested readers.
	
	{\bf Question 1} Let $b\geq 2$, $n\geq 1$, and $d=(b-1,n)$. Can one find a necessary and sufficient condition for the equality
		\begin{equation}\label{eqQ}
		\Out{s_b,n}= \left\{u\in\mathbb{N}\ \vert\ u\equiv 0\Mod{d}\right\}?
	\end{equation}	
	When $d=1$, we obtain such a condition in Theorem \ref{THM3}, but when $d>1$, the problem is open. By Theorems \ref{thm1} and \ref{THM2}, it is easy to see that if $d$ is divisible by $s_b(n)$, then \eqref{eqQ} holds. Is the divisibility $s_b(n)\mid d$ also a necessary condition for \eqref{eqQ}?
	
	{\bf Question 2} By Examples \ref{Ex4} and \ref{Ex5}, we know that \eqref{eqQ}  may or may not hold. In the case that $d>1$ and $\Out{s_b,n}$ is not equal to the set on the right-hand side of \eqref{eqQ}, can one completely determine the elements of $\Out{s_b,n}$? 
	
	{\bf Question 3} We obtain in Theorem \ref{THM2} that $\Out{s_b,n}$ contains every multiple of $s_b(n)$. Does $\Out{s_b,n}$ contain an infinite number of positive integers that are not divisible by $s_b(n)$? When $d=1$ and $s_b(n)\geq 2$, the answer is yes. Can we determine a simple necessary and sufficient condition for the infinitely many elements in $A\cap\Out{s_b,n}$ where $A$ is the set of positive integers that are not divisible by $s_b(n)$? In that case, does $\Out{s_b,n}$ has a natural density?
	
	{\bf Question 4} For a cofinite proper subset $A$ of $\mathbb{N}$, we call the largest integer in $\mathbb{N}\setminus A$ the Frobenius number of $A$. By Theorems \ref{thm1} and \ref{THM3}, we know that $\Out{s_b,n}$ is a cofinite proper subset of $\mathbb{N}$ when $(b-1,n)=1$ and there exists a prime $p$ that divides $n$ but does not divide $b$. Can one determine the Frobenius number of $\Out{s_b,n}$ in this case? The reader can find more information on numerical sets an Frobenius numbers, for example in the book by Alfons\'{i}n \cite{Alfonsin} and in the introduction of the article by Guhl et al. \cite{Guhl}.
	
	{\bf Question 5} We obtain some basic properties of $\Out{g,n}$ where $g=s_b, \tau, \omega, \Omega$ but we do not have any result on the set of $u\in\mathbb{N}$ that there exists an arc from $u$ to $n$. So for each arithmetic function $g$ and $n\in\mathbb{N}$, let
	\[\text{In}(g,n)= \left\{u\in\mathbb{N}\ \vert\ u\xrightarrow{g}n\right\}.\]
	What are the properties of $\text{In}(g,n)$? Is $\text{In}(g,n)$ infinite? Can one determine the set of all positive integers $n$ such that $\text{In}(g,n)=\mathbb{N}$ or a cofinite subset of $\mathbb{N}$? Does such an $n$ exist?
	
	{\bf Question 6} We can extend the $g$-arc from the multiple of $n$ to an arithmetic progression $r\Mod{n}$ by defining that there exists a $(g,r)$-arc form $n$ to $u$ if there exists $N\in\mathbb{N}$ such that $N\equiv r\Mod{n}$ and $g(N)=u$. Let
	\[\Out{g,n,r}=\{u\in\mathbb{N}\mid \text{ there exists a } (g,r)\text{-arc from } n  \text{ to } u\},\]
	\[\text{In}(g,n,r)=\{u\in\mathbb{N}\mid \text{ there exists a } (g,r)\text{-arc from } u \text{ to } n\}.\]
	So if $r=0$, then this is the $g$-arc that we study in this article. When $0<r<n$, what are the results analogous to our theorems?
	
	{\bf Question 7} We may consider the $g$-arc from $n$ to $\infty$ too. A sequence $(a_m)_{m\geq 1}$ is a $g$-arc from $n$ to $\infty$ if $a_1=n$, $a_m\rightarrow\infty$ as $m\rightarrow\infty$, and there exists a $g$-arc from $a_m$ to $a_{m+1}$ for every $m\geq 1$. Nevertheless, without any restriction on the number of steps, this problem may not be interesting. We say that there exists a $k$-bounded $g$-arc from $n$ to $u$ if there exists $N\in\mathbb{N}$ such that $N\equiv 0\Mod{n}$, $N\leq kn$, and $g(N)=u$. A sequence $(a_m)_{m\geq 1}$ is a $k$-bounded $g$-arc from $n$ to $\infty$ if $a_1=n$, $a_m\rightarrow\infty$ as $m\rightarrow\infty$, and there exists a $k$-bounded $g$-arc from $a_m$ to $a_{m+1}$ for every $m\in\mathbb{N}$. What are the results corresponding to our theorems for $k$-bounded $g$-arc? If $n$ and $k$ are given, does there exist a $k$-bounded $g$-arc from $n$ to $\infty$ when $g=s_b, \tau, \omega, \Omega$, or other arithmetic functions?
	
	{\bf Question 8} We say that $n$ and $u$ are $g$-friends if there exists a $g$-arc from $n$ to $u$ and there exists a $g$-arc from $u$  to $n$. For each $g=s_b,\tau,\omega,\Omega$, can we determine all pairs $(n,u)$ which are $g$-friends?
	
	{\bf Question 9} A triangle in a directed graph $G=(V,E)$ is a triple $(v_1,v_2,v_3)$ such that there exist directed edges from $v_1$ to $v_2$, $v_2$ to $v_3$, and $v_3$ to $v_1$. If $n\geq 3$, an $n$-polygon in $G$ is an $n$-tuple $(v_1,v_2,\ldots,v_n)$ such that there are directed edges connecting from $v_n$ to $v_1$ and from $v_i$ to $v_{i+1}$ for each $i=1,2,\ldots, n-1$. Suppose $g$ is an arithmetic function, $V=\mathbb{N}$, and there exists a directed edge from $n$ to $u$ if there exists a $g$-arc from $n$ to $u$. Are there infinitely many triangles or $n$-polygons in $G$? 
	
	{\bf Question 10} Let $e\geq 1$ and $b\geq 2$ be integers, and let $S_{e,b}:\mathbb{N}\rightarrow\mathbb{N}$ be defined as follows: if $n=(a_ka_{k-1}\cdots a_0)_b=\sum\limits_{0\leq i\leq k}a_ib^{i}$ is the $b$-adic expansion of $n$ where $a_k\neq 0$ and $0\leq a_i<b$ for all $i$, then
	\[S_{e,b}(n)=a_{k}^{e}+a_{k-1}^{e}+\cdots+a_{0}^{e}.\]
	The function $S_{e,b}$ is called an $(e,b)$-happy function, and it has been studied by many mathematicians, see for example in Guy's book \cite[Chapter E34]{Guy}, and in the articles by El-Sedy and Siksek \cite{Sedy}, Grundman and Teeple \cite{Grundman2}, Gilmer \cite{Gilmer}, Chase \cite{Chase}, Noppakaew et al. \cite{Noppakaew}, and Subwattanachai and Pongsriiam \cite{Subwattanachai}. In particular, if $e=1$, then $S_{e,b}=s_b$ is the sum of digits function. So the study of $S_{e,b}$-arc may lead to an interesting generalization of $s_b$-arc. What are the corresponding results to our theorems if we replace $s_b$ by $S_{e,b}$?
	
	\noindent {\bf Comments and Acknowledgements.} We already mentioned the work of various mathematicians on graph labelling and coprime graphs \cite{Erdos1, Erdos2, Sarkozy, Vaidya, Berliner, Lee, Berkove}, which inspired us to write this article. Another inspiration came from the last problem in Thailand Mathematical Olympiad TMO 2022 \  \cite{TMO} in which the first author participated as a teacher. After TMO 2022 ended, the first author sent a question to the third author who had heard of Gelfond's theorem on sum of digits function $s_b(n)$, which contains a hint to an answer to the problem, and we decided to extend this problem without the use of Gelfond's theorem. The interested reader can find more information about Gelfond's theorem, for example, in Morgenbesser's diploma thesis \cite{Morgen}, references therein, and many other articles in the literature.
	
	After some modification of languages, the contestants in TMO 2022 were asked about the cofiniteness of $\Out{s_{10},n}$, which are covered and extended to $\Out{s_{b},n}$ for any $b\geq 2$ and $n\geq 1$ in (ii) of Theorem \ref{thm1} in this article. We would like to thank and show the support to the organizers, teachers, students, and sponsors of TMO by writing this article. We hope that this will motivate some students to learn more about mathematics.

	Pongsriiam's research project is funded jointly by the Faculty of Science Silpakorn University and the National Research Council of Thailand (NRCT), grant number NRCT5-RSA63021-02. He is also supported by the Tosio Kato Fellowship given by the Mathematical Society of Japan during his visit at Nagoya University in July 2022 to July 2023.

\end{document}